\documentclass{siamltex}

\usepackage[latin9]{inputenc}
\usepackage{float}
\usepackage{amsmath}
\usepackage{amssymb}
\usepackage{graphicx}

\usepackage{xcolor}

\makeatletter

\usepackage{algorithmic}
\floatstyle{ruled}
\newfloat{algorithm}{tbp}{loa}
\providecommand{\algorithmname}{Algorithm}
\floatname{algorithm}{\protect\algorithmname}

\makeatother

\begin{document}

\title{Error norm estimates for the block conjugate gradient algorithm\footnotemark[1]}

\author{G\'{e}rard Meurant\footnotemark[2] and Petr Tich\'{y}\footnotemark[3]}

\footnotetext[2] {Paris, France ({\tt gerard.meurant@gmail.com})}

\footnotetext[3]{Faculty of Mathematics and Physics, Charles University,
Sokolovska 83, Prague, 18675, Czech Republic ({\tt petr.tichy@mff.cuni.cz})}

\footnotetext[1]{Version of \today}

\maketitle

\begin{abstract}
In the book [Meurant and Tich\'{y}, SIAM, 2024] we discussed the estimation
of error norms in the conjugate gradient (CG) algorithm for solving
linear systems $Ax=b$ with a symmetric positive definite matrix $A$,
where $b$ and $x$ are vectors. In this paper, we generalize the
most important formulas for estimating the $A$-norm of the error
to the block case. First, we discuss in detail the derivation of various
variants of the block CG (BCG) algorithm from the block Lanczos algorithm.
We then consider BCG and derive the related block Gauss and block
Gauss-Radau quadrature rules. We show how to obtain lower and upper
bounds on the $A$-norm of the error of each system, both in terms
of the quantities computed in BCG and in terms of the underlying block
Lanczos algorithm. Numerical experiments demonstrate the behavior
of the bounds in practical computations.
\end{abstract}

%
%

\section{Introduction}

The block conjugate gradient (BCG) algorithm was introduced by D.P.~O'Leary
\cite{ol1980b} in 1980 for solving linear systems with several right-hand
sides
\begin{equation}
AX=B,\label{eq-AXB}
\end{equation}
where $A$ is a real symmetric positive definite matrix of order $n$,
and $X$ and $B$ are $n\times m$ real matrices with $m\ll n$.

In \cite{ol1980b}, BCG was considered as a special case of the block
biconjugate gradient algorithm that was stated without any derivation.
Then, the mathematical properties of BCG were derived and proved.
We have not found any precise derivation of BCG in the literature, except in \cite{btk2015}.
In this paper, we first show with details how to derive BCG from the
block Lanczos algorithm introduced by J.~Cullum and W.E.~Donath
\cite{cud1974} in 1974 and G.H.~Golub and R.~Underwood \cite{gu1977}
in 1977. We also consider relations between these two algorithms that are useful for
our main goal which is to derive lower and upper bounds
for the $A$-norms of the errors for the approximate solutions corresponding
to each column of the block right-hand side. This is done by considering
block Gauss and block Gauss-Radau quadrature rules. Bounds can also
be obtained for variants of BCG proposed in \cite{ol1980b,dub2001}.

\smallskip{}

In section~\ref{s-sec2} we recall the block Lanczos algorithm. The
BCG algorithm is obtained from the block Lanczos algorithm in section~\ref{s-sec3}.
Relations between the (matrix) coefficients of the block Lanczos
and BCG algorithms are derived in section~\ref{s-sec4}. How to measure
the errors is considered in section~\ref{s-sec5}. Section~\ref{s-sec6}
shows how to use a block Gauss quadrature rule to obtain lower bounds
for the $A$-norms of the errors. Upper bounds are obtained in section~\ref{s-sec7}
from a block Gauss-Radau quadrature rule. The problem of rank deficiency
and an algorithm by A.~Dubrulle \cite{dub2001} are considered in
section~\ref{s-sec8}. Numerical experiments showing the effectiveness
of these techniques for obtaining bounds are discussed in section~\ref{s-sec9}.

\smallskip{}

Matrices and block vectors are denoted by upper case letters. The
$n\times mn$ matrices which are collections of block vectors are
denoted by calligraphic letters. The ``small'' blocks of order $m$
are denoted by upper case Greek letters.

\section{The block Lanczos algorithm}

\label{s-sec2}

The block Lanczos algorithm constructs an orhonormal basis of the
block Krylov subspaces,
\[
\mathcal{K}_{k}(A,V)\equiv\mathrm{colspan}\{V,AV,\dots,A^{k-1}V\},
\]
where $V$ is an $n\times m$ given matrix. Here $\mathrm{colspan}$
denotes the span of the columns of all block vectors $V,AV,\dots,A^{k-1}V$.
To derive the standard block algorithms we will assume that the considered
block Krylov subspaces $\mathcal{K}_{k}(A,V)$ have full dimension
$km$.

Let $V_{0}$ be the $n\times m$ zero matrix and let $V_{1}\Gamma_{0}=V$
be the QR factorization of $V$. The matrices $V_{k}$ (block vectors)
are defined by a block three-term recurrence,
\[
W_{k}=AV_{k}-V_{k}\Omega_{k}-V_{k-1}\Gamma_{k-1}^{T},\quad V_{k+1}\Gamma_{k}=W_{k},
\]
where $\Omega_{k}$ and $\Gamma_{k}$ are square matrices of order
$m$. The rightmost equality is the QR factorization of $W_{k}$. So,
$V_{k+1}$ has orthonormal columns and $\Gamma_{k}$ is upper triangular.
The block vectors are computed such that
\[
V_{i}^{T}V_{j}=0,\ i\ne j,\quad V_{i}^{T}V_{i}=I_{m},
\]
where $I_m$ is the identity matrix of order $m$.
The block Lanczos algorithm is described in Algorithm~\ref{alg-BLanczos}.

\smallskip{}

\begin{algorithm}[th]
\caption{Block Lanczos }
\label{alg-BLanczos}

\begin{algorithmic}[1]

\STATE \textbf{input} $A$, $V$

\STATE $V_{0}=0$

\STATE $V_{1}\Gamma_{0}=V$

\FOR{$k=1,2,\dots$}

\STATE $W=AV_{k}-V_{k-1}\Gamma_{k-1}^{T}$

\STATE $\Omega_{k}=V_{k}^{T}W$

\STATE $W_{k}=W-V_{k}\Omega_{k}$

\STATE $V_{k+1}\Gamma_{k}=W_{k}$

\ENDFOR

\end{algorithmic}
\end{algorithm}

We stack the block vectors $V_{i}$ in a matrix
\[
{\cal V}_{k}=\left(V_{1},V_{2,}\dots,V_{k}\right).
\]
Then, we can write the recurrences, up to block vector $V_{k+1}$
as
\begin{equation}
A{\cal V}_{k}={\cal V}_{k}T_{k}+{\cal G}_{k},\quad{\cal G}_{k}=V_{k+1}\Gamma_{k}E_{k}^{T},\label{eq-brecur}
\end{equation}
where $E_{k}$ is the $km\times m$ block vector which is zero except
for the last block which is~$I_m$. Using
the Knonecker product notation, we can also define
\begin{equation}
E_{j}\equiv e_{j}\otimes I_{m},\quad j=1,\dots,k,\label{def:E}
\end{equation}
where $e_{j}\in\mathbb{R}^{k}$ is the $j$th column of the identity
matrix $I_{k}$. In relation~\eqref{eq-brecur}, $T_{k}$ is a symmetric
block tridiagonal matrix of order $km$,
\[
T_{k}=\begin{pmatrix}\Omega_{1} & \Gamma_{1}^{T}\\
\Gamma_{1} & \ddots & \ddots\\
 & \ddots & \ddots & \Gamma_{k-1}^{T}\\
 &  & \Gamma_{k-1} & \Omega_{k}
\end{pmatrix}.
\]
Since we assume that $A$ is positive definite, $T_{k}$ is positive
definite as well, and we can factorize the matrix $T_{k}$ as
\begin{equation}
T_{k}=\widetilde{L}_{k}\begin{pmatrix}\Delta_{1}^{-1}\\
 & \ddots\\
 &  & \Delta_{k}^{-1}
\end{pmatrix}\widetilde{L}_{k}^{T}=L_{k}\widetilde{L}_{k}^{T}\label{eq:LU}
\end{equation}
with
\[
\widetilde{L}_{k}=\begin{pmatrix}\Delta_{1}\\
\Gamma_{1} & \ddots\\
 & \ddots & \ddots\\
 &  & \Gamma_{k-1} & \Delta_{k}
\end{pmatrix},\quad L_{k}=\begin{pmatrix}I_{m}\\
\Pi_{1} & \ddots\\
 & \ddots & \ddots\\
 &  & \Pi_{k-1} & I_{m}
\end{pmatrix},
\]
where $\Pi_{j}=\Gamma_{j}\Delta_{j}^{-1},$ $j=1,\dots,k-1$. The
symmetric and positive definite matrices $\Delta_{i}\in\mathbb{R}^{m\times m}$
are computed as
\begin{equation}
\Delta_{1}=\Omega_{1},\quad\Delta_{j}=\Omega_{j}-\Gamma_{j-1}\Delta_{j-1}^{-1}\Gamma_{j-1}^{T},\ \ j=2,\dots,k.\label{eq:LDLT}
\end{equation}

We can now use the block Lanczos algorithm to solve a system of linear
equations with several right-hand sides $B$ using the orthogonal
residual approach. In general, we do not need the assumption that
$A$ is positive definite, but we need to assume that $T_{k}$ is
nonsingular. We can consider an approximate solution $X_{k}\in\mathbb{R}^{n\times m}$
of equation~\eqref{eq-AXB} in the form
\begin{equation}
X_{k}=X_{0}+{\cal V}_{k}Z_{k},\quad X_{0}\in\mathbb{R}^{n\times m}\textnormal{\ given},\label{eq-defX}
\end{equation}
where $Z_{k}\in\mathbb{R}^{km\times m}$ is to be determined. We will
require that ${\cal V}_{k}^{T}R_{k}=0$, where $R_{k}$ is the residual
block vector $R_{k}=B-AX_{k}$ that can be computed as
\[
R_{k}=R_{0}-A{\cal V}_{k}Z_{k},\quad R_{0}=B-AX_{0}.
\]
Using relation~\eqref{eq-brecur}, we obtain
\[
R_{k}=R_{0}-{\cal V}_{k}T_{k}Z_{k}-{\cal G}_{k}Z_{k}.
\]
Left multiplication with ${\cal V}_{k}^{T}$ yields
\[
0={\cal V}_{k}^{T}R_{0}-T_{k}Z_{k}.
\]
The last term ${\cal V}_{k}^{T}{\cal G}_{k}Z_{k}$ is zero because
of the orthogonality relations.

It is convenient to define $V_{1}$ in relation to $R_{0}$. With
our hypothesis $R_{0}$ is of full column rank, i.e., $R_{0}^{T}R_{0}$
is symmetric positive definite. So, we can define a factorization
\[
R_{0}^{T}R_{0}=\Phi_{0}^{T}\Phi_{0},
\]
and $V_{1}=R_{0}\Phi_{0}^{-1}$. The matrix $\Phi_{0}$ of order $m$
can be an upper Cholesky factor obtained from a QR factorization of
$R_{0}$, or, a square root of $R_{0}^{T}R_{0}$. Then,
\[
{\cal V}_{k}^{T}R_{0}=\begin{pmatrix}V_{1}^{T}R_{0}\\
0\\
\vdots\\
0
\end{pmatrix}=\begin{pmatrix}\Phi_{0}\\
0\\
\vdots\\
0
\end{pmatrix}=E_{1}\Phi_{0},
\]
where $E_{1}$ is defined as in \eqref{def:E}. With the choice $V_{1}=R_{0}\Phi_{0}^{-1}$,
the block vector $Z_{k}$ in \eqref{eq-defX} is obtained by solving
a block tridiagonal linear system
\[
T_{k}Z_{k}=E_{1}\Phi_{0},
\]
where only the first block of the right-hand side is nonzero. It yields
\begin{equation}
X_{k}=X_{0}+{\cal V}_{k}T_{k}^{-1}E_{1}\Phi_{0},\label{eq-XkV}
\end{equation}
and
\begin{eqnarray*}
R_{k} & = & R_{0}-A{\cal V}_{k}Z_{k}\\
 & = & R_{0}-({\cal V}_{k}T_{k}+{\cal G}_{k})T_{k}^{-1}E_{1}\Phi_{0}\\
 & = & V_{1}\Phi_{0}-{\cal V}_{k}E_{1}\Phi_{0}-{\cal G}_{k}T_{k}^{-1}E_{1}\Phi_{0}\\
 & = & -V_{k+1}\Gamma_{k}[T_{k}^{-1}E_{1}]_{k}\Phi_{0},
\end{eqnarray*}
where $[T_{k}^{-1}E_{1}]_{k}$ is the last block of the first block
column of $T_{k}^{-1}$. Denoting
\[
\Phi_{k}=(-1)^{k-1}\Gamma_{k}[T_{k}^{-1}E_{1}]_{k}\Phi_{0},
\]
we obtain
\begin{equation}
V_{k+1}=(-1)^{k}R_{k}\Phi_{k}^{-1},\quad\mbox{and}\quad R_{k}^{T}R_{k}=\Phi_{k}^{T}\Phi_{k}.\label{eq:Lanczos_res}
\end{equation}
We chose this sign for $\Phi_{k}$ for compatibility with the scalar
case ($m=1$). The relation between the residuals and the basis block
vectors shows that $R_{i}^{T}R_{j}=0$, for $i\ne j$. Note that if
$A$ is positive definite, then it is straightforward to show using
the factorization \eqref{eq:LU} that
\[
[T_{k}^{-1}E_{1}]_{k}=(-1)^{k-1}\Delta_{k}^{-1}\Pi_{k-1}\cdots\Pi_{1}
\]
so that
\begin{equation}
\Phi_{k}=\Pi_{k}\cdots\Pi_{1}\Phi_{0}=\Pi_{k}\Phi_{k-1}.\label{eq:Phi}
\end{equation}

\smallskip{}

For computing the approximate solutions $X_{k}$ with relation~\eqref{eq-XkV},
we have to store all the block vectors $V_{i}$, $i=1,\dots,k$. This
can be avoided when $A$ is positive definite by using the factorization
\eqref{eq:LU} of $T_{k}$, leading to the block conjugate gradient
algorithm discussed in the next section.

\section{Block CG from block Lanczos}

\label{s-sec3}

From now on we assume that $A$ is positive definite. For the derivation
of BCG, we assume that there are no rank deficiency problems.
We will discuss this issue in section~\ref{s-sec8}. First
of all, let us consider a block diagonal matrix
\[
D_{k}=\mathrm{diag}(\Psi_{1},\dots,\Psi_{k})
\]
with nonsingular diagonal blocks $\Psi_{i}$ to be chosen later to
simplify some formulas. By inserting the matrix $D_{k}^{-1}D_{k}$,
we can write the factorization \eqref{eq:LU} of $T_{k}$ as
\[
T_{k}=L_{k}D_{k}^{-1}D_{k}\widetilde{L}_{k}^{T}=L_{k}D_{k}^{-1}U_{k}
\]
with
\[
U_{k}=\begin{pmatrix}\Psi_{1}\Delta_{1} & \Psi_{1}\Gamma_{1}^{T}\\
 & \ddots & \ddots\\
 &  & \ddots & \Psi_{k-1}\Gamma_{k-1}^{T}\\
 &  &  & \Psi_{k}\Delta_{k}
\end{pmatrix}.
\]

We introduce a new $n\times km$ matrix
\[
{\cal P}_{k}=\left(P_{0},P_{1},\dots,P_{k-1}\right)={\cal V}_{k}U_{k}^{-1},
\]
where the block vectors $P_{i}$ are $n\times m$. The block Lanczos
iterates are
\begin{eqnarray*}
X_{k} & = & X_{0}+{\cal V}_{k}T_{k}^{-1}E_{1}\Phi_{0}\\
 & = & X_{0}+{\cal V}_{k}U_{k}^{-1}D_{k}L_{k}^{-1}E_{1}\Phi_{0}\\
 & = & X_{0}+{\cal P}_{k}D_{k}L_{k}^{-1}E_{1}\Phi_{0}.
\end{eqnarray*}
For $k=1$, we get $X_{1}=X_{0}+P_{0}\Psi_{1}\Phi_{0}.$ Since $D_{k}$
is block diagonal and $L_{k}^{-1}$ lower block triangular, for $k>1$ we can
split the second term on the right-hand side in two pieces
${\cal P}_{k-1}D_{k-1}L_{k-1}^{-1}[E_{1}]_{1:k-1}\Phi_{0}$ and $P_{k-1}\Psi_{k}[L_{k}^{-1}E_{1}]_{k}\Phi_{0}$,
where $[L_{k}^{-1}E_{1}]_{k}$ is the last block of $L_{k}^{-1}E_{1}$.
Since the matrix $L_{k}$ is lower block bidiagonal with diagonal
blocks which are identity matrices of order $m$, we obtain
\[
[L_{k}^{-1}E_{1}]_{k}=(-1)^{k-1}\Pi_{k-1}\cdots\Pi_{1}.
\]
Hence, we can write the $k$th iterate as
\[
X_{k}=X_{k-1}+P_{k-1}\Upsilon_{k-1},
\]
with
\[
\Upsilon_{k-1}=(-1)^{k-1}\Psi_{k}\Pi_{k-1}\cdots\Pi_{1}\Phi_{0}=(-1)^{k-1}\Psi_{k}\Phi_{k-1},
\]
where we have used \eqref{eq:Phi} in the last equality. The relation
for $X_{k}$ yields a relation for the block residual vectors,
\begin{equation}
R_{k}=R_{k-1}-AP_{k-1}\Upsilon_{k-1}.\label{eq-Rk}
\end{equation}

Let us choose $\Psi_{k}=(-1)^{k-1}\Phi_{k-1}^{-1}\Delta_{k}^{-1}$.
Then, the coefficient
\begin{equation}
\Upsilon_{k-1}=\Phi_{k-1}^{-1}\Delta_{k}^{-1}\Phi_{k-1}\label{eq-cUpPhi}
\end{equation}
is similar to the inverse of the last diagonal block of the factorization
of $T_{k}$.

Now, we need to show how to compute $P_{k}$. Comparing the last block
columns in ${\cal P}_{k+1}U_{k+1}={\cal V}_{k+1}$ and using \eqref{eq:Lanczos_res}
we obtain
\[
P_{k-1}\Psi_{k}\Gamma_{k}^{T}+P_{k}\Psi_{k+1}\Delta_{k+1}=V_{k+1}=(-1)^{k}R_{k}\Phi_{k}^{-1}.
\]
We multiply from the right by $\Delta_{k+1}^{-1}\Psi_{k+1}^{-1}=(-1)^{k}\Phi_{k}$
to get
\[
P_{k}=R_{k}-P_{k-1}\Psi_{k}\Gamma_{k}^{T}(-1)^{k}\Phi_{k}.
\]
 From the definition of $\Psi_{k}$ and using \eqref{eq:Phi} we obtain
\[
(-1)^{k-1}\Psi_{k}\Gamma_{k}^{T}\Phi_{k}=\Phi_{k-1}^{-1}\Delta_{k}^{-1}\Gamma_{k}^{T}\Phi_{k}=\Phi_{k-1}^{-1}\Phi_{k-1}^{-T}\Phi_{k}^{T}\Phi_{k}.
\]
Finally, using the relation between $R_{k}$ and $\Phi_{k}$, it follows
that
\[
P_{k}=R_{k}+P_{k-1}\Xi_{k},\quad\Xi_{k}=(R_{k-1}^{T}R_{k-1})^{-1}(R_{k}^{T}R_{k}).
\]
Using this relation and the orthogonality of the block residual vectors,
we obtain by induction that $P_{i}^{T}R_{j}=0$, $i<j$. Moreover,
$P_{k-1}^{T}R_{k-1}=R_{k-1}^{T}R_{k-1}$. Observing that
\[
0=P_{k-1}^{T}R_{k}=P_{k-1}^{T}R_{k-1}-P_{k-1}^{T}AP_{k-1}\Upsilon_{k-1},
\]
we can compute the matrix coefficient $\Upsilon_{k-1}$ as
\[
\Upsilon_{k-1}=(P_{k-1}^{T}AP_{k-1})^{-1}(P_{k-1}^{T}R_{k-1})=(P_{k-1}^{T}AP_{k-1})^{-1}(R_{k-1}^{T}R_{k-1}).
\]

The standard BCG algorithm is given as Algorithm~\ref{alg-BCG}.

\smallskip{}

\begin{algorithm}[ht]
\caption{BCG}
\label{alg-BCG}

\begin{algorithmic}[1]

\STATE \textbf{input} $A$, $B$, $X_{0}$

\STATE $R_{0}=B-AX_{0}$

\STATE $P_{0}=R_{0}$

\FOR{$k=1,\dots$ until convergence}

\STATE $\Upsilon_{k-1}=(P_{k-1}^{T}AP_{k-1})^{-1}(R_{k-1}^{T}R_{k-1})$

\STATE $X_{k}=X_{k-1}+P_{k-1}\Upsilon_{k-1}$

\STATE $R_{k}=R_{k-1}-AP_{k-1}\Upsilon_{k-1}$

\STATE $\Xi_{k}=(R_{k-1}^{T}R_{k-1})^{-1}(R_{k}^{T}R_{k})$

\STATE $P_{k}=R_{k}+P_{k-1}\Xi_{k}$

\ENDFOR

\end{algorithmic}
\end{algorithm}

\smallskip{}

Of course, we have assumed that $P_{k-1}^{T}AP_{k-1}$ and $R_{k-1}^{T}R_{k-1}$
are nonsingular. As we wrote above, we will come back to this point in section~\ref{s-sec8}.

\smallskip{}

We observe that D.P.~O'Leary \cite{ol1980b} added to BCG other (scaling)
matrices $\Sigma_{i}$ of order $m$ that can be chosen at our will;
see Algorithm~\ref{alg-BCG-OL}. For instance, in the formula to
compute $P_{k}$, $\Sigma_{k}$ can be the inverse of the R factor
of a QR factorization of $R_{k}+P_{k-1}\Xi_{k}$. Even though Algorithms~\ref{alg-BCG}
and \ref{alg-BCG-OL} must produce the same iterates and residuals
in exact arithmetic, introducing these matrices was intended to improve
the numerical properties of the algorithm.

\smallskip{}

\begin{algorithm}[ht]
\caption{D.P.~O'Leary's BCG (OL-BCG)}
\label{alg-BCG-OL}

\begin{algorithmic}[1]

\STATE \textbf{input} $A$, $B$, $X_{0}$

\STATE $R_{0}=B-AX_{0}$

\STATE $P_{0}=R_{0}\Sigma_{0}$

\FOR{$k=1,\dots$ until convergence}

\STATE $\Upsilon_{k-1}=(P_{k-1}^{T}AP_{k-1})^{-1}\Sigma_{k-1}^{T}(R_{k-1}^{T}R_{k-1})$

\STATE $X_{k}=X_{k-1}+P_{k-1}\Upsilon_{k-1}$

\STATE $R_{k}=R_{k-1}-AP_{k-1}\Upsilon_{k-1}$

\STATE $\Xi_{k}=\Sigma_{k-1}^{-1}(R_{k-1}^{T}R_{k-1})^{-1}(R_{k}^{T}R_{k})$

\STATE $P_{k}=(R_{k}+P_{k-1}\Xi_{k})\Sigma_{k}$

\ENDFOR

\end{algorithmic}
\end{algorithm}

\smallskip{}

A derivation of BCG from orthogonality conditions is described in
S.~Birk's Ph.D.~thesis \cite{btk2015}.

\section{Relations between the coefficients of the block Lanczos and BCG algorithms}

\label{s-sec4}

Even though the relations between the coefficients of block Lanczos
and BCG are stated in \cite{tms2024}, let us see briefly what they
are with the notation of this paper since we will need them in the
next sections.

\smallskip{}

From relation~\eqref{eq-cUpPhi}, we have
\begin{equation}
\Delta_{k}=\Phi_{k-1}\Upsilon_{k-1}^{-1}\Phi_{k-1}^{-1},\label{eq:Delta}
\end{equation}
and from~\eqref{eq:Phi}, the block Lanczos coefficient $\Gamma_{k}$
is
\begin{equation}
\Gamma_{k}=\Phi_{k}\Phi_{k-1}^{-1}\Delta_{k}=\Phi_{k}\Upsilon_{k-1}^{-1}\Phi_{k-1}^{-1}.\label{eq-Gamma}
\end{equation}
For the other block Lanczos coefficient $\Omega_{k}$ we have
\[
\Omega_{k}=\Delta_{k}+\Gamma_{k-1}\Delta_{k-1}^{-1}\Gamma_{k-1}^{T};
\]
see \eqref{eq:LDLT}. Using \eqref{eq:Delta} and \eqref{eq-Gamma},
$\Omega_{k}$ can be written in the form
\begin{equation}
\Omega_{k}=\Phi_{k-1}\Upsilon_{k-1}^{-1}\Phi_{k-1}^{-1}+\Phi_{k-1}\Phi_{k-2}^{-1}\Phi_{k-2}^{-T}\Upsilon_{k-2}^{-1}\Phi_{k-1}^{T}.\label{eq-Omega1}
\end{equation}
Finally, by observing that
\[
\Phi_{k-2}^{-1}\Phi_{k-2}^{-T}=\Xi_{k-1}\Phi_{k-1}^{-1}\Phi_{k-1}^{-T},
\]
we obtain
\begin{equation}
\Omega_{k}=\Phi_{k-1}\Upsilon_{k-1}^{-1}\Phi_{k-1}^{-1}+[\Phi_{k-1}\Xi_{k-1}\Phi_{k-1}^{-1}]\,[\Phi_{k-1}^{-T}\Upsilon_{k-2}^{-1}\Phi_{k-1}^{T}].\label{eq-Omega}
\end{equation}
Up to products with $\Phi_{k-1}$, this is a relation which looks like
what we have for the Lanczos and conjugate gradient (CG) algorithms
when $m=1$ (see \cite{gm2006}) since the three matrices that are
involved in the right-hand side are respectively similar to $\Upsilon_{k-1}^{-1}$,
$\Xi_{k-1}$, and~$\Upsilon_{k-2}^{-1}$.

\section{How to measure the error}

\label{s-sec5}

In CG, it is natural to consider the $A$-norm of the error $\Vert x-x_{k}\Vert_{A}$
since it is minimized at each iteration. In BCG, it is the trace of
\[
\mathfrak{E}_{k}=(X-X_{k})^{T}A(X-X_{k})
\]
which is minimized; see \cite[Theorem~2]{ol1980b}. In more detail,
denoting by $x_{k}^{(i)}$, $r_{k}^{(i)}$, $p_{k}^{(i)}$, $i=1,\dots,m$,
the columns of $X_{k}$, $R_{k}$, $P_{k}$ respectively, it holds
that
\begin{equation}
x_{k}^{(i)}\in x_{0}^{(i)}+\mathcal{K}_{k}(A,R_{0}),\quad r_{k}^{(i)}\perp\mathcal{K}_{k}(A,R_{0}).\label{eq:ortogonality}
\end{equation}
Moreover,
\[
p_{k}^{(i)}\perp_{A}\mathcal{K}_{k}(A,R_{0}).
\]
As a consequence, $x_{k}^{(i)}$ minimizes
\[
(y-x^{(i)})^{T}A(y-x^{(i)})=\Vert y-x^{(i)}\Vert_{A}^{2}
\]
over all vectors $y\in x_{0}^{(i)}+\mathcal{K}_{k}(A,R_{0})$. To
generalize the estimates known from CG \cite{mt2024}, we need the following theorem.

\smallskip{}

\begin{theorem}
\label{t-Thetak} Consider Algorithm~\ref{alg-BCG} and assume that
the block vectors $R_{i}$, $i=0,\dots,k-1$, are of full column rank.
Then it holds that
\begin{equation}
\mathfrak{E}_{k-1}=\Theta_{k-1}+\mathfrak{E}_{k},\label{eq-Theta}
\end{equation}
and the matrix
\begin{equation}
\Theta_{k-1}\equiv(R_{k-1}^{T}R_{k-1})\Upsilon_{k-1}\label{eq:defTheta}
\end{equation}
 is symmetric and positive definite.
\end{theorem}
\smallskip{}

\begin{proof}
Under the assumptions of the theorem, the vectors $X_{i}$, $R_{i}$,
and $P_{i}$, $i=0,\dots,k$ are well defined. It is easy to check
the following relation
\[
\mathfrak{E}_{k-1}-\mathfrak{E}_{k}=(X_{k}-X_{k-1})^{T}A(X_{k}-X_{k-1})+2(X-X_{k})^{T}A(X_{k}-X_{k-1}).
\]
Using $A(X-X_{k})=R_{k}$ and $X_{k}-X_{k-1}=P_{k-1}\Upsilon_{k-1}$
we get
\[
\mathfrak{E}_{k-1}-\mathfrak{E}_{k}=\Upsilon_{k-1}^{T}P_{k-1}^{T}AP_{k-1}\Upsilon_{k-1}+2R_{k}^{T}P_{k-1}\Upsilon_{k-1}.
\]
The second term of the right-hand side is zero because $P_{k-1}^{T}R_{k}=0$.
The first term can be written as
\begin{eqnarray*}
\Upsilon_{k-1}^{T}P_{k-1}^{T}AP_{k-1}\Upsilon_{k-1} & = & (R_{k-1}^{T}R_{k-1})(P_{k-1}^{T}AP_{k-1})^{-1}(P_{k-1}^{T}AP_{k-1})\Upsilon_{k-1}\\
 & = & (R_{k-1}^{T}R_{k-1})\Upsilon_{k-1}.
\end{eqnarray*}
Since $P_{k-1}\Upsilon_{k-1}$ has full column rank, $\Theta_{k-1}\equiv(R_{k-1}^{T}R_{k-1})\Upsilon_{k-1}$
is symmetric and positive definite.
\end{proof}
\smallskip{}

Taking only the diagonal entries in \eqref{eq-Theta} we find out
that
\[
\textnormal{diag}(\mathfrak{E}_{k-1})=\textnormal{diag}(\Theta_{k-1})+\textnormal{diag}(\mathfrak{E}_{k}),
\]
i.e., for $i=1,\dots,m$ it holds that
\[
\Vert x^{(i)}-x_{k-1}^{(i)}\Vert_{A}^{2}=\left(\Theta_{k-1}\right){}_{i,i}+\Vert x^{(i)}-x_{k}^{(i)}\Vert_{A}^{2},
\]
where $(\Theta_{k-1}){}_{i,i}$ is the entry at position $(i,i)$
of the matrix $\Theta_{k-1}$. Since the matrix $\Theta_{k-1}$ has
positive diagonal entries (it is symmetric and positive definite),
we obtain easily computable lower bounds
\[
\Vert x^{(i)}-x_{k-1}^{(i)}\Vert_{A}^{2}\ge\left(\Theta_{k-1}\right){}_{i,i},\quad i=1,\dots,m,
\]
where the inequality is strict whenever $\Vert x^{(i)}-x_{k}^{(i)}\Vert_{A}>0$.
In other words, the diagonal entries of $\Theta_{k-1}$ are lower bounds
on the squares of the $A$-norms of the columns of the error block
vectors at iteration $k-1$.

As for CG, we can sum relations~\eqref{eq-Theta} for several consecutive
values of $k$ to obtain better bounds; see \cite{mt2024}. In detail,
for a given integer $\ell\geq k$, we get
\begin{equation}
\mathfrak{E}_{k-1}=\sum_{j=k-1}^{\ell-1}\Theta_{k-1}+\mathfrak{E}_{\ell}.\label{eq:better}
\end{equation}

\section{The block Gauss quadrature rule}

\label{s-sec6}

In this section we show that relation~\eqref{eq-Theta} can be considered
as a block Gauss quadrature rule. For details on block Gauss quadrature rules, see \cite{gme2010}.

Mathematically, the block Lanczos algorithm must terminate at or before
$\lceil n/m\rceil$ iterations. Assume that it terminates after $q$
iterations with $A{\cal V}_{q}={\cal V}_{q}T_{q}$. We proceed as
for CG in \cite[Theorem~1.3]{mt2024}. We have
\[
X=X_{0}+{\cal V}_{q}T_{q}^{-1}\underline{E}_{1}\Phi_{0},\quad X_{k}=X_{0}+{\cal V}_{k}T_{k}^{-1}E_{1}\Phi_{0},
\]
where $E_{1}$ (resp.~$\underline{E}_{1}$) is a block vector $km\times m$
(resp.~$qm\times m$) which is zero except for the first block which
is $I_{m}$; see \eqref{def:E}. It yields
\[
X-X_{k}={\cal V}_{q}\left[T_{q}^{-1}\underline{E}_{1}-\begin{pmatrix}T_{k}^{-1}E_{1}\\
0
\end{pmatrix}\right]\Phi_{0}.
\]
Then $\mathfrak{E}_{k}=(X-X_{k})^{T}A(X-X_{k})$ can be written as
\[
\Phi_{0}^{T}\left[T_{q}^{-1}\underline{E}_{1}-\begin{pmatrix}T_{k}^{-1}E_{1}\\
0
\end{pmatrix}\right]^{T}{\cal V}_{q}^{T}A{\cal V}_{q}\left[T_{q}^{-1}\underline{E}_{1}-\begin{pmatrix}T_{k}^{-1}E_{1}\\
0
\end{pmatrix}\right]\Phi_{0},
\]
and, since ${\cal V}_{q}^{T}A{\cal V}_{q}=T_{q}$, we obtain
\begin{equation}
\mathfrak{E}_{k}=\Phi_{0}^{T}\left([T_{q}^{-1}]_{1,1}-[T_{k}^{-1}]_{1,1}\right)\Phi_{0}.\label{eq:Thetak}
\end{equation}
This can be considered as a block Gauss quadrature rule, see \cite{gme2010}.
The matrix $[T_{k}^{-1}]_{1,1}$ is the approximation of $[T_{q}^{-1}]_{1,1}$,
and $\Phi_{0}^{-T}\mathfrak{E}_{k}\Phi_{0}^{-1}$ is the remainder.
At iteration $k$ we do not know $T_{q}$ but, like in CG, we can
take a difference
\[
\mathfrak{E}_{k}-\mathfrak{E}_{k+d}=\Phi_{0}^{T}\left([T_{k+d}^{-1}]_{1,1}-[T_{k}^{-1}]_{1,1}\right)\Phi_{0},
\]
with a given integer $d>0$ which is named the delay. The block $[T_{k}^{-1}]_{1,1}$
can be computed with the factorization of $T_{k}$.

Let us consider $d=1$. As in \cite[p.~34]{mt2024}, we write $T_{k+1}$
as
\[
T_{k+1}=\begin{pmatrix}T_{k} & E_{k}\Gamma_{k}^{T}\\
\Gamma_{k}E_{k}^{T} & \Omega_{k+1}
\end{pmatrix}.
\]
We would like to compute the $(1,1)$ block of $T_{k+1}^{-1}$. It
is given by the $(1,1)$ block of the inverse of the Schur complement
\[
S=T_{k}-(E_{k}\Gamma_{k}^{T})\Omega_{k+1}^{-1}(\Gamma_{k}E_{k}^{T}).
\]
To obtain $S^{-1}$, we use the Sherman-Morrison-Woodbury formula,
\[
S^{-1}=T_{k}^{-1}+(T_{k}^{-1}E_{k})\Gamma_{k}^{T}[\Omega_{k+1}-\Gamma_{k}(E_{k}^{T}T_{k}^{-1}E_{k})\Gamma_{k}^{T}]^{-1}\Gamma_{k}(E_{k}^{T}T_{k}^{-1}).
\]
The $(1,1)$ block of $S^{-1}$ (which is equal to $[T_{k+1}^{-1}]_{1,1}$)
can be written as
\[
[T_{k+1}^{-1}]_{1,1}=[T_{k}^{-1}]_{1,1}+[Y]_{1}\Gamma_{k}^{T}[\Omega_{k+1}-\Gamma_{k}[Y]_{k}\Gamma_{k}^{T}]^{-1}\Gamma_{k}[Y]_{1}^{T},
\]
where $Y=T_{k}^{-1}E_{k}$, and $[Y]_{1}$ and $[Y]_{k}$ are the
first and the $k$th blocks of $Y$. The blocks of $Y$ can be computed
using the block factorization \eqref{eq:LU} of $T_{k}$. It is straightforward
to see that
\[
[Y]_{k}=\Delta_{k}^{-1},\quad[Y]_{1}=(-1)^{k-1}\Pi_{1}^{T}\cdots\Pi_{k-1}^{T}\Delta_{k}^{-1}.
\]
Note that using $\Pi_{k}^{T}=\Delta_{k}^{-1}\Gamma_{k}^{T}$ and \eqref{eq:Phi}
we have
\[
(-1)^{k-1}[Y]_{1}\Gamma_{k}^{T}=\Pi_{1}^{T}\cdots\Pi_{k-1}^{T}\Pi_{k}^{T}=\Phi_{0}^{-T}\Phi_{k}^{T}
\]
and from \eqref{eq:LDLT} it follows that
\[
\Omega_{k+1}-\Gamma_{k}\Delta_{k}\Gamma_{k}^{T}=\Delta_{k+1}.
\]
Therefore,
\begin{eqnarray*}
[T_{k+1}^{-1}]_{1,1}-[T_{k}^{-1}]_{1,1} & = & \Phi_{0}^{-T}\Phi_{k}^{T}\Delta_{k+1}^{-1}\Phi_{k}\Phi_{0}^{-1}\\
 & = & \Phi_{0}^{-T}\Phi_{k}^{T}(\Phi_{k}\Upsilon_{k}\Phi_{k}^{-1})\Phi_{k}\Phi_{0}^{-1}\\
 & = & \Phi_{0}^{-T}\Phi_{k}^{T}\Phi_{k}\Upsilon_{k}\Phi_{0}^{-1}
\end{eqnarray*}
so that
\[
\Phi_{0}^{T}\left([T_{k+1}^{-1}]_{1,1}-[T_{k}^{-1}]_{1,1}\right)\Phi_{0}=\Phi_{k}^{T}\Phi_{k}\Upsilon_{k}=(R_{k}^{T}R_{k})\Upsilon_{k}=\Theta_{k}.
\]
This is consistent with relation~\eqref{eq-Theta}. It shows that
the lower bounds given by this relation are obtained from a block
Gauss quadrature rule.

\section{The block Gauss-Radau quadrature rule}

\label{s-sec7}

To obtain upper bounds, we have to use a block Gauss-Radau quadrature
rule with a prescribed node. We proceed as in \cite{gme2010}, page
109 and following. For the block Lanczos algorithm we have matrix
orthogonal polynomials $p_{j}$ associated to $T_{k}$ satisfying
a block three-term recurrence,
\begin{equation}
\lambda p_{j-1}(\lambda)=p_{j}(\lambda)\Gamma_{j}+p_{j-1}(\lambda)\Omega_{j}+p_{j-2}(\lambda)\Gamma_{j-1}^{T},\label{eq:recp}
\end{equation}
with initial conditions $p_{0}(\lambda)\equiv I_{m},\ p_{-1}(\lambda)\equiv0$; see \cite{gme2010}.
Writing the recurrences \eqref{eq:recp} for $j=1,\dots,k$ we have
\[
\begin{pmatrix}\Omega_{1} & \Gamma_{1}^{T}\\
\Gamma_{1} & \ddots & \ddots\\
 & \ddots & \ddots & \Gamma_{k-1}^{T}\\
 &  & \Gamma_{k-1} & \Omega_{k}
\end{pmatrix}\begin{pmatrix}p_{0}^{T}(\lambda)\\
\vdots\\
p_{k-2}^{T}(\lambda)\\
p_{k-1}^{T}(\lambda)
\end{pmatrix}=\lambda\begin{pmatrix}p_{0}^{T}(\lambda)\\
\vdots\\
p_{k-2}^{T}(\lambda)\\
p_{k-1}^{T}(\lambda)
\end{pmatrix}-\begin{pmatrix}0\\
\vdots\\
0\\
\Gamma_{k}^{T}p_{k}^{T}(\lambda)
\end{pmatrix}E_{k}^{T}.
\]
The nodes of the block Gauss quadrature rule are the eigenvalues of
$T_{k}$ and the weights are the first $m$ components of the eigenvectors;
see \cite{gme2010}. For the block Gauss-Radau rule, we would like
a real positive $\mu$ to be a multiple eigenvalue (with multiplicity
$m$) of an extended matrix $T_{k+1}^{(\mu)}$,
\[
T_{k+1}^{(\mu)}=\begin{pmatrix}\Omega_{1} & \Gamma_{1}^{T}\\
\Gamma_{1} & \Omega_{2} & \Gamma_{2}^{T}\\
 & \ddots & \ddots & \ddots\\
 &  & \Gamma_{k-1} & \Omega_{k} & \Gamma_{k}^{T}\\
 &  &  & \Gamma_{k} & \Omega_{k+1}^{(\mu)}
\end{pmatrix}.
\]
Such a $\Omega_{k+1}^{(\mu)}$ is determined in the following preposition.

\smallskip{}

\begin{proposition}
\label{p-Omega} Suppose that $\mu$ is not an eigenvalue of $T_{k}$.
Then $\mu$ is a multiple eigenvalue (with multiplicity $m$) of the
extended matrix $T_{k+1}^{(\mu)}$ if and only if
\begin{equation}
\Omega_{k+1}^{(\mu)}=\mu I_{m}+\Gamma_{k}[(T_{k}-\mu I)^{-1}]_{k,k}\Gamma_{k}^{T}.\label{eq:Omegamu}
\end{equation}
\end{proposition}

\begin{proof}
The necessary and sufficient condition for $\mu$ to be an eigenvalue
with multiplicity $m$ of the extended matrix $T_{k+1}^{(\mu)}$ is
\[
\mu p_{k}(\mu)-p_{k}(\mu)\Omega_{k+1}^{(\mu)}-p_{k-1}(\mu)\Gamma_{k}^{T}=0.
\]
Since $\mu$ is not an eigenvalue of $T_{k}$, the matrix $p_{k}(\mu)$
is nonsingular, and
\[
\Omega_{k+1}^{(\mu)}=\mu I_{m}-p_{k}^{-1}(\mu)p_{k-1}(\mu)\Gamma_{k}^{T},
\]
so we only need to evaluate $p_{k-1}^{T}(\mu)p_{k}^{-T}(\mu)$ using
\eqref{eq:recp}. Substituting $\mu$ into \eqref{eq:recp} and multiplying
each recurrence by $p_{k}(\mu)^{-1}$ from the left, we obtain a system
of linear equations
\[
(T_{k}-\mu I)\begin{pmatrix}p_{0}(\mu)^{T}p_{k}(\mu)^{-T}\\
\vdots\\
p_{k-1}(\mu)^{T}p_{k}(\mu)^{-T}
\end{pmatrix}=-\begin{pmatrix}0\\
\vdots\\
0\\
\Gamma_{k}^{T}
\end{pmatrix}=-E_{k}\Gamma_{k}^{T},
\]
so that
\[
p_{k-1}(\mu)^{T}p_{k}(\mu)^{-T}=-E_{k}^{T}(T_{k}-\mu I)^{-1}E_{k}\Gamma_{k}^{T}.
\]
Finally,
\begin{eqnarray*}
\Omega_{k+1}^{(\mu)} & = & \mu I_{m}-p_{k}^{-1}(\mu)p_{k-1}(\mu)\Gamma_{k}^{T}\\
 & = & \mu I_{m}+\Gamma_{k}E_{k}^{T}(T_{k}-\mu I)^{-1}E_{k}\Gamma_{k}^{T}.
\end{eqnarray*}
For the derivation of a block Gauss-Radau quadrature rule, one can
also see K.~Lund's Ph.D.~thesis \cite[p.~104]{lund2018}.
\end{proof}
\smallskip{}

We would like to express everything in terms of the quantities computed
in BCG. To do that we have to consider the block factorization of
the shifted matrix $T_{k}-\mu I$ as in \eqref{eq:LU}. We will proceed
analogously as in \cite[Chapter~4]{mt2024}.
\begin{lemma}
\label{l-Delta} Let $\overline{\Delta}_{k+1}^{(\mu)}$ (resp.~$\Delta_{k+1}^{(\mu)}$),
$k=0,\dots,j$, be the diagonal blocks of the factorizations of $T_{j+1}-\mu I$
(resp.~$T_{j+1}^{(\mu)}$). Then,
\[
\Delta_{k+1}^{(\mu)}=\Delta_{k+1}-\overline{\Delta}_{k+1}^{(\mu)}.
\]
\end{lemma}

\begin{proof}
From \eqref{eq:LDLT}, the diagonal blocks of $T_{j+1}-\mu I$ are
defined by the block recurrence
\[
\overline{\Delta}_{k+1}^{(\mu)}=\Omega_{k+1}-\mu I_{m}-\Gamma_{k}[\overline{\Delta}_{k}^{(\mu)}]^{-1}\Gamma_{k}^{T}
\]
with $\overline{\Delta}_{1}^{(\mu)}=\Omega_{1}-\mu I_{m}$. From the
factorization of $T_{j+1}-\mu I$ it easily follows that
\[
E_{k}^{T}(T_{k}-\mu I)^{-1}E_{k}=[\overline{\Delta}_{k}^{(\mu)}]^{-1},
\]
so that using \eqref{eq:Omegamu},
\[
\Omega_{k+1}^{(\mu)}=\mu I_{m}+\Gamma_{k}[\overline{\Delta}_{k}^{(\mu)}]^{-1}\Gamma_{k}^{T}.
\]
We also consider the factorizations of $T_{k+1}$ and $T_{k+1}^{(\mu)}$
for which we have
\[
\Delta_{k+1}=\Omega_{k+1}-\Gamma_{k}\Delta_{k}^{-1}\Gamma_{k}^{T},\quad\Delta_{k+1}^{(\mu)}=\Omega_{k+1}^{(\mu)}-\Gamma_{k}\Delta_{k}^{-1}\Gamma_{k}^{T}.
\]

These relations lead to
\begin{eqnarray*}
\overline{\Delta}_{k+1}^{(\mu)} & = & \left(\Omega_{k+1}\right)-\mu I_{m}-\Gamma_{k}[\overline{\Delta}_{k}^{(\mu)}]^{-1}\Gamma_{k}^{T}\\
 & = & \Delta_{k+1}+\left(\Gamma_{k}\Delta_{k}^{-1}\Gamma_{k}^{T}\right)-\mu I_{m}-\Gamma_{k}[\overline{\Delta}_{k}^{(\mu)}]^{-1}\Gamma_{k}^{T}\\
 & = & \Delta_{k+1}+\left(\Omega_{k+1}^{(\mu)}\right)-\Delta_{k+1}^{(\mu)}-\mu I_{m}-\Gamma_{k}[\overline{\Delta}_{k}^{(\mu)}]^{-1}\Gamma_{k}^{T}\\
 & = & \Delta_{k+1}-\Delta_{k+1}^{(\mu)}.
\end{eqnarray*}
Therefore, the last diagonal block of the factorization of $T_{k+1}^{(\mu)}$
is a modification of the last diagonal block of the factorization
of $T_{k+1}$,
\[
\Delta_{k+1}^{(\mu)}=\Delta_{k+1}-\overline{\Delta}_{k+1}^{(\mu)},
\]
which proves the claim.
\end{proof}
\smallskip{}

Before discussing the computation of the block Gauss-Radau bounds,
we prove the following algebraic lemma, which will be used later.\smallskip{}

\begin{lemma}
\label{l-inverses}Let $G$ and $H$ be two real nonsingular matrices
such that $H-G$ is nonsingular. Then it holds that
\[
\left(G^{-1}-H^{-1}\right)^{-1}=G(H-G)^{-1}G+G.
\]
Moreover, if $G$ and $H-G$ are symmetric and positive definite,
then also
\[
G^{-1}-H^{-1}
\]
is symmetric and positive definite.
\end{lemma}
\smallskip{}

\begin{proof}
Using a simple algebraic manipulation we obtain
\begin{eqnarray*}
G(H-G)^{-1}G+G & = & G(H-G)^{-1}G+G(H-G)^{-1}(H-G)\\
 & = & G(H-G)^{-1}H\\
 & = & (H^{-1}\left(H-G\right)G^{-1})^{-1}\\
 & = & (G^{-1}-H^{-1})^{-1}.
\end{eqnarray*}

If $G$ and $H-G$ are symmetric and positive definite, then
$H$ is positive definite. Using the above identity
\[
G((H-G)^{-1}+G^{-1})G=\left(G^{-1}-H^{-1}\right)^{-1}
\]
and the fact that the matrix
\[
(H-G)^{-1}+G^{-1}
\]
is symmetric and positive definite, we find out by Sylvester's law
of inertia that $G^{-1}-H^{-1}$ has positive eigenvalues.
\end{proof}
\smallskip{}

To compute the block Gauss-Radau bounds we need the following result.

\smallskip{}

\begin{theorem}
\label{t-Gamma} Denoting $\Upsilon_{j}^{(\mu)}\equiv\Phi_{j}^{-1}[\Delta_{j+1}^{(\mu)}]^{-1}\Phi_{j}$,
$j=0,\dots,k$, it holds that $\Upsilon_{0}^{(\mu)}=\mu^{-1}I_{m}$
and
\begin{eqnarray}
\Upsilon_{k}^{(\mu)} & = & [\mu\,(\Upsilon_{k-1}^{(\mu)}-\Upsilon_{k-1})+\Xi_{k}]^{-1}\,(\Upsilon_{k-1}^{(\mu)}-\Upsilon_{k-1})\label{eq-Gammamu}\\
 & = & [\mu\,(\Theta_{k-1}^{(\mu)}-\Theta_{k-1})+R_{k}^{T}R_{k}]^{-1}(\Theta_{k-1}^{(\mu)}-\Theta_{k-1})\label{eq:Gammamu1}
\end{eqnarray}
for $k>0$, where
\begin{equation}
\Theta_{k-1}^{(\mu)}\equiv(R_{k-1}^{T}R_{k-1})\,\Upsilon_{k-1}^{(\mu)}.\label{eq:defC}
\end{equation}
\end{theorem}

\begin{proof}
The initial condition comes from $\Delta_{1}^{(\mu)}=\mu I_{m}$.
As in the proof of some previous results, we can get the relation
\begin{eqnarray*}
\Delta_{k+1}^{(\mu)} & = & \Omega_{k+1}^{(\mu)}-\Gamma_{k}\Delta_{k}^{-1}\Gamma_{k}^{T}\\
 & = & \mu I_{m}+\Gamma_{k}[\overline{\Delta}_{k}^{(\mu)}]^{-1}\Gamma_{k}^{T}-\Gamma_{k}\Delta_{k}^{-1}\Gamma_{k}^{T}\\
 & = & \mu I_{m}+\Gamma_{k}\left([\Delta_{k}-\Delta_{k}^{(\mu)}]^{-1}-\Delta_{k}^{-1}\right)\Gamma_{k}^{T}.
\end{eqnarray*}
Since $\Pi_{k}=\Gamma_{k}\Delta_{k}^{-1}$, we can replace $\Gamma_{k}$
and obtain
\begin{eqnarray*}
\Delta_{k+1}^{(\mu)}-\mu I_{m} & = & \Pi_{k}\left(\Delta_{k}[\Delta_{k}-\Delta_{k}^{(\mu)}]^{-1}\Delta_{k}-\Delta_{k}\right)\Pi_{k}^{T}\\
 & = & \Pi_{k}[(\Delta_{k}^{(\mu)})^{-1}-\Delta_{k}^{-1}]^{-1}\Pi_{k}^{T},
\end{eqnarray*}
where we used the first part of Lemma~\ref{l-inverses}.

Using \eqref{eq:Phi}, \eqref{eq-cUpPhi}, and $\Upsilon_{k-1}^{(\mu)}=\Phi_{k-1}^{-1}[\Delta_{k}^{(\mu)}]^{-1}\Phi_{k-1}$,
we obtain
\begin{eqnarray*}
\Delta_{k+1}^{(\mu)}-\mu I_{m} & = & \Phi_{k}\Phi_{k-1}^{-1}[\Phi_{k-1}\Upsilon_{k-1}^{(\mu)}\Phi_{k-1}^{-1}-\Phi_{k-1}\Upsilon_{k-1}\Phi_{k-1}^{-1}]^{-1}\Phi_{k-1}^{-T}\Phi_{k}^{T}\\
 & = & \Phi_{k}[\Upsilon_{k-1}^{(\mu)}-\Upsilon_{k-1}]^{-1}\Phi_{k-1}^{-1}\Phi_{k-1}^{-T}\Phi_{k}^{T}
\end{eqnarray*}
and hence
\[
\Phi_{k}^{-1}\Delta_{k+1}^{(\mu)}\Phi_{k}-\mu I_{m}=[\Upsilon_{k-1}^{(\mu)}-\Upsilon_{k-1}]^{-1}\Phi_{k-1}^{-1}\Phi_{k-1}^{-T}\Phi_{k}^{T}\Phi_{k}.
\]
Since $\Phi_{k}^{-1}\Delta_{k+1}^{(\mu)}\Phi_{k}=[\Upsilon_{k}^{(\mu)}]^{-1}$
and $\Phi_{k-1}^{-1}\Phi_{k-1}^{-T}\Phi_{k}^{T}\Phi_{k}=\Xi_{k}$,
we get
\begin{eqnarray*}
[\Upsilon_{k}^{(\mu)}]^{-1} & = & \mu I_{m}+[\Upsilon_{k-1}^{(\mu)}-\Upsilon_{k-1}]^{-1}\Xi_{k}\\
 & = & [\Upsilon_{k-1}^{(\mu)}-\Upsilon_{k-1}]^{-1}\left(\mu[\Upsilon_{k-1}^{(\mu)}-\Upsilon_{k-1}]+\Xi_{k}\right).
\end{eqnarray*}
The relation \eqref{eq-Gammamu} follows by taking the inverse of
both sides. Finally, using $\Xi_{k}=(R_{k-1}^{T}R_{k-1})^{-1}(R_{k}^{T}R_{k})$
and the definition of $\Theta_{k-1}^{(\mu)}$, we obtain \eqref{eq:Gammamu1}.
\end{proof}
\smallskip{}

What we have done above is the block equivalent of \texttt{stqds}
and \texttt{dstqds} algorithms for CG; see \cite[Chapter~4]{mt2024}. For the
Gauss-Radau rule, we have to consider
\[
\Phi_{0}^{T}([(T_{k+1}^{(\mu)})^{-1}]_{1,1}-[T_{k}^{-1}]_{1,1})\Phi_{0}=(R_{k}^{T}R_{k})\Upsilon_{k}^{(\mu)}=\Theta_{k}^{(\mu)},
\]
where the above equality is obtained similarly as for the Gauss rule,
except that we have to replace $\Omega_{k+1}$, $\Delta_{k+1}$, and
$\Upsilon_{k}$ respectively by $\Omega_{k+1}^{(\mu)}$, $\Delta_{k+1}^{(\mu)}$,
and $\Upsilon_{k}^{(\mu)}$. Now, we would like to prove that we can
obtain upper bounds from this result.

In the following we will assume that $0<\mu<\lambda_{1},$ where $\lambda_{1}$
is the smallest eigenvalue of $A$. We first prove a result about positive definiteness
of some blocks that will be used later.
\begin{lemma}
Let $\mu$ be such that \textup{$0<\mu<\lambda_{1}$}. Assume that the matrices $R_{k}$,
$k=0,\dots,\ell-1,$ have full column rank. Then it holds that the
blocks $\Delta_{k+1}^{(\mu)}$ as well as $\Theta_{k}^{(\mu)}-\Theta_{k}$
defined in \eqref{eq:defC} and \eqref{eq:defTheta} are symmetric
and positive definite matrices.
\end{lemma}
\smallskip{}

\begin{proof}
We prove the positive definiteness of $\Delta_{k+1}^{(\mu)}$ by induction.
For $k=0$,
\[
\Delta_{1}^{(\mu)}=\Delta_{1}-\overline{\Delta}_{k}^{(\mu)}=\mu I_{m.}
\]

Suppose now that $\Delta_{k}^{(\mu)}=\Delta_{k}-\overline{\Delta}_{k}^{(\mu)}$
is positive definite. Since $\Delta_{k}$ and $\overline{\Delta}_{k}^{(\mu)}$
are positive definite, we know by Lemma~\ref{l-inverses} that
\[
[\overline{\Delta}_{k}^{(\mu)}]^{-1}-\Delta_{k}^{-1}
\]
is positive definite. Therefore,
\[
\Delta_{k+1}^{(\mu)}=\mu I_{m}+\Gamma_{k}\left([\overline{\Delta}_{k}^{(\mu)}]^{-1}-\Delta_{k}^{-1}\right)\Gamma_{k}^{T}
\]
is positive definite as well.

We now prove the result for $\Theta_{k}^{(\mu)}-\Theta_{k}$. Since
$R_{k}^{T}R_{k}=\Phi_{k}^{T}\Phi_{k}$,
\[
\Upsilon_{k}=\Phi_{k}^{-1}\Delta_{k+1}^{-1}\Phi_{k},\quad\mbox{and}\quad\Upsilon_{k}^{(\mu)}=\Phi_{k}^{-1}[\Delta_{k+1}^{(\mu)}]^{-1}\Phi_{k}
\]
we get
\begin{eqnarray*}
\Theta_{k}^{(\mu)}-\Theta_{k} & = & \Phi_{k}^{T}\left([\Delta_{k+1}^{(\mu)}]^{-1}-\Delta_{k+1}^{-1}\right)\Phi_{k}.
\end{eqnarray*}
Since $\Delta_{k+1}^{(\mu)}=\Delta_{k+1}-\overline{\Delta}_{k+1}^{(\mu)}$
as well as $\Delta_{k+1}$ and $\overline{\Delta}_{k+1}^{(\mu)}$
are positive definite, $\Theta_{k}^{(\mu)}-\Theta_{k}$ is positive
definite by Lemma~\ref{l-inverses} and Sylvester's law of inertia.
\end{proof}
\smallskip{}

We finally prove results which imply that we can obtain upper
bounds.

\smallskip{}

\begin{theorem}
\label{t-GRup} Let $0<\mu<\lambda_{1}$ where $\lambda_{1}$ is the
smallest eigenvalue of $A$. Assume that the matrices $R_{k-1}$,
$k=1,\dots,\ell,$ have full column rank. Then it holds that
\begin{equation}
\left(\Theta_{k-1}^{(\mu)}-\mathfrak{E}_{k-1}\right)=\sum_{j=k}^{\ell}B_{j}^{(\mu)}+\left(\Theta_{\ell}^{(\mu)}-\mathfrak{E}_{\ell}\right)\label{eq:keymatrix}
\end{equation}
where the matrices
\[
B_{j}^{(\mu)}\equiv(\Theta_{j-1}^{(\mu)}-\Theta_{j-1})-\Theta_{j}^{(\mu)},
\]
$j=1,\dots,\ell$, are positive definite. If $R_{\ell}=0$, then the
matrices $\Theta_{k-1}^{(\mu)}-\mathfrak{E}_{k-1}$ are positive definite.
\end{theorem}
\smallskip{}

\begin{proof}
We get
\begin{eqnarray*}
\left(\Theta_{k-1}^{(\mu)}-\mathfrak{E}_{k-1}\right)-\left(\Theta_{k}^{(\mu)}-\mathfrak{E}_{k}\right) & = & \Theta_{k-1}^{(\mu)}-\left(\mathfrak{E}_{k-1}-\mathfrak{E}_{k}\right)-\Theta_{k}^{(\mu)}\\
 & = & (\Theta_{k-1}^{(\mu)}-\Theta_{k-1})-\Theta_{k}^{(\mu)}=\,B_{k}^{(\mu)}
\end{eqnarray*}
so that \eqref{eq:keymatrix} holds.

We prove now positive definiteness of $B_{k}^{(\mu)}$. From \eqref{eq:Gammamu1}
we have
\begin{eqnarray*}
\Theta_{k-1}^{(\mu)}-\Theta_{k-1} & = & [\mu\,(\Theta_{k-1}^{(\mu)}-\Theta_{k-1})+R_{k}^{T}R_{k}]\Upsilon_{k}^{(\mu)}\\
 & = & \mu\,(\Theta_{k-1}^{(\mu)}-\Theta_{k-1})\Upsilon_{k}^{(\mu)}+\Theta_{k}^{(\mu)}
\end{eqnarray*}
so that
\begin{eqnarray*}
B_{k}^{(\mu)} & = & \mu\,(\Theta_{k-1}^{(\mu)}-\Theta_{k-1})\,\Upsilon_{k}^{(\mu)}\\
 & = & \mu\,(\Theta_{k-1}^{(\mu)}-\Theta_{k-1})\,[\mu\,(\Theta_{k-1}^{(\mu)}-\Theta_{k-1})+R_{k}^{T}R_{k}]^{-1}(\Theta_{k-1}^{(\mu)}-\Theta_{k-1}).
\end{eqnarray*}
Recall that $(\Theta_{k-1}^{(\mu)}-\Theta_{k-1})$ is symmetric and
positive definite. Therefore, applying Sylvester's law of inertia,
we just have to consider the matrix
\[
\mu\,(\Theta_{k-1}^{(\mu)}-\Theta_{k-1})+R_{k}^{T}R_{k}.
\]
It is the sum of a positive definite matrix $\mu\,(\Theta_{k-1}^{(\mu)}-\Theta_{k-1})$
and the matrix $R_{k}^{T}R_{k}$ which is at least semi-positive definite.
Hence, this matrix is positive definite so that the matrix $B_{k}^{(\mu)}$
has positive eigenvalues.

Finally, if $R_{\ell}=0$, then $\mathfrak{E}_{\ell}=0$ and using
\eqref{eq:keymatrix}, $\Theta_{k-1}^{(\mu)}-\Theta_{k-1}$ is a sum
of positive definite matrices.
\end{proof}
\smallskip{}

Mathematically, if the underlying block Lanczos algorithm terminates
after $q$ iterations with $A{\cal V}_{q}={\cal V}_{q}T_{q}$, then
$R_{q}=0$, and the matrices
\begin{equation}
\Theta_{k-1}^{(\mu)}-\mathfrak{E}_{k-1},\quad k=1,\dots,q,\label{eq:keym}
\end{equation}
are positive definite; see Theorem~\ref{t-GRup}. 
In finite precision
computations, one can expect that the relation
\eqref{eq:keymatrix} holds up to some small inaccuracy, despite the
loss of global orthogonality. However, one cannot usually expect that
$R_{\ell}=0$ for some index $\ell$. Using \eqref{eq:keymatrix},
it is then important to ask whether the symmetric and positive definite
matrix
\[
\sum_{j=k}^{\ell}B_{j}^{(\mu)}
\]
dominates the symmetric matrix
\[
\left(\Theta_{\ell}^{(\mu)}-\mathfrak{E}_{\ell}\right)
\]
which can eventually be indefinite, in particular if some of the error
norms reaches the level of maximum attainable accuracy. This question
is not easy to answer; see \cite[Chapter~7]{mt2024} for the analysis
of the vector case. In our numerical experience, the matrices \eqref{eq:keym}
are positive definite during the computations until the level of maximum
attainable accuracy is reached.

\smallskip{}

If the matrices \eqref{eq:keym} are positive definite, then the diagonal
entries of $\Theta_{k-1}^{(\mu)}$ give upper bounds of the squares
of the $A$-norms of the errors. However, we can obtain better upper
bounds at the previous iterations, as in \cite{mt2024} for CG. For
instance, a better upper bound at iteration $k-1$ is obtained using
\eqref{eq:better}, when replacing $\mathfrak{E}_{\ell}$ by $\Theta_{\ell}^{(\mu)}$.

\medskip{}

In Algorithm~\ref{alg-BCG-OL}, the coefficient $\Upsilon_{k-1}$
is computed as
\[
\Upsilon_{k-1}=(P_{k-1}^{T}AP_{k-1})^{-1}\Sigma_{k-1}^{T}(R_{k-1}^{T}R_{k-1}).
\]
If we use the same technique as above to compute the Gauss lower bound,
it is easy to see that the estimate at iteration $k-1$ must be
\[
(R_{k-1}^{T}R_{k-1})\Sigma_{k-1}\Upsilon_{k-1}.
\]
This can be seen in a different way. According to O'Leary \cite{ol1980b},
$X_{k}$ and $R_{k}$ are invariant regarding the choice of the $\Sigma_{j}$'s,
since they are uniquely determined by the orthogonality conditions
\eqref{eq:ortogonality}. Nevertheless, the $P_{j}$'s are not invariant,
and we can look at the relations between the coefficients of both
variants. Let us denote the variables of O'Leary's variant with a
tilde. Then, we assume that $\tilde{X}_{k}=X_{k}$ and $\tilde{R}_{k}=R_{k}$.
From the equality of the iterates, we obtain
\[
P_{k-1}\Upsilon_{k-1}=\tilde{P}_{k-1}\tilde{\Upsilon}_{k-1}.
\]
Since the residual block vectors are the same, we have
\[
\tilde{\Xi}_{k}=\Sigma_{k-1}^{-1}\Xi_{k}.
\]
This again justifies the computation of the Gauss lower bound. We
have $\tilde{P}_{0}\Sigma_{0}=R_{0}=P_{0}$, which yields $\tilde{P}_{0}=P_{0}\Sigma_{0}$.
Let us assume that $\tilde{P}_{j}=P_{j}\Sigma_{j}$ for $j\le k-1$.
Then,
\begin{eqnarray*}
\tilde{P}_{k} & = & (R_{k}+\tilde{P}_{k-1}\tilde{\Xi}_{k})\Sigma_{k}\\
 & = & (R_{k}+\tilde{P}_{k-1}\Sigma_{k-1}^{-1}\Xi_{k})\Sigma_{k}\\
 & = & (R_{k}+P_{k-1}\Xi_{k})\Sigma_{k}.
\end{eqnarray*}
It shows that $\tilde{P}_{k}=P_{k}\Sigma_{k}$. Now, we can get a
relation for the other coefficient,
\begin{eqnarray*}
\tilde{\Upsilon}_{k-1} & = & (\tilde{P}_{k-1}^{T}A\tilde{P}_{k-1})^{-1}\Sigma_{k-1}^{T}(R_{k-1}^{T}R_{k-1})\\
 & = & (\Sigma_{k-1}^{T}P_{k-1}^{T}AP_{k-1}\Sigma_{k-1})^{-1}\Sigma_{k-1}^{T}(R_{k-1}^{T}R_{k-1})\\
 & = & \Sigma_{k-1}^{-1}\Upsilon_{k-1}.
\end{eqnarray*}
Hence, $\tilde{\Upsilon}_{k-1}=\Sigma_{k-1}^{-1}\Upsilon_{k-1}$.

\smallskip{}

To obtain the Gauss-Radau bounds, it is enough to compute $\Upsilon_{k}^{(\mu)}$
as in \eqref{eq-Gammamu} where we replace $\Upsilon_{k-1}$ by $\Sigma_{k-1}\tilde{\Upsilon}_{k-1}$.
The lower and upper bounds of the $A$-norms of the errors derived
in this section are also valid when BCG is used with a preconditioner;
see \cite{ol1980b,tms2024}.

\section{Rank deficiency}

\label{s-sec8}

So far, we have assumed that there was no rank deficiency in the matrices
$R_{k}$ and $P_{k}$. However, in practical computations, these matrix
may become rank deficient. It can happen when one of the linear systems
converges faster than the other ones, but may also happen during the
iterations without any convergence. In that case, $R_{k}^{T}R_{k}$
becomes only semi-positive definite and we cannot compute the BCG
coefficients. The remedy proposed in \cite{ol1980b} is to use deflation,
that is, to decrease the block size when rank deficiency occurs, but
this has a negative impact on the efficiency of the method. It is
not easy to find which vectors are responsible for the problems. A
rank-revealing QR factorization may be needed. On this problem, see
also D.~Ruiz \cite{rui1992}, M.~Arioli, I.S.~Duff, D.~Ruiz, and
M.~Sadkane \cite{adrs1995}, A.A.~Nikishin and A.Yu.~Yeremin \cite{niye1995},
S.~Birk and A.~Frommer \cite{BiFr2014}, and H.~Ji and Y.~Li \cite{jili2017}.

\smallskip{}

Another remedy was proposed by A.~Dubrulle (1935-2016) \cite{dub2001}.
He used a change of variable to handle rank deficiency and described
several variants. Let us consider an algorithm that is called Dubrulle-R (DR-BCG)
in \cite{tms2024}. It corresponds to BCGrQ in \cite{dub2001}. In
\cite{tms2024} the algorithm is derived from O'Leary's BCG. However,
it can be derived directly from BCG by introducing another change
of variable. For the derivation of Dubrulle's variants from BCG, see
also \cite{btk2015}.

\smallskip{}

Considering Algorithm~\ref{alg-BCG}, we start from a QR factorization
of the residual block vector, $R_{k}=Q_{k}\widehat{\Phi}_{k}$, where
the matrix $\widehat{\Phi}_{k}$ is of order $m$. The inverse of
that matrix (which may eventually be singular in case of rank deficiency)
will appear in the derivation of the algorithm, but not in the final
algorithm. Even if $R_{k}$ is rank deficient, $Q_{k}$ has orthonormal
columns. We also introduce $S_{k}=P_{k}\widehat{\Phi}_{k}$. The relation
for $P_{k}$ becomes
\[
S_{k}\widehat{\Phi}_{k}=Q_{k}\widehat{\Phi}_{k}+S_{k-1}\widehat{\Phi}_{k-1}\Xi_{k}.
\]
Since $\Xi_{k}=\widehat{\Phi}_{k-1}^{-1}\widehat{\Phi}_{k-1}^{-T}\widehat{\Phi}_{k}\widehat{\Phi}_{k}$,
it yields
\begin{eqnarray*}
S_{k} & = & Q_{k}+S_{k-1}\widehat{\Phi}_{k-1}\Xi_{k}\widehat{\Phi}_{k}^{-1}\\
 & = & Q_{k}+S_{k-1}(\widehat{\Phi}_{k}\widehat{\Phi}_{k-1}^{-1})^{T}.
\end{eqnarray*}
The other BCG coefficient is
\begin{eqnarray*}
\Upsilon_{k-1} & = & [\widehat{\Phi}_{k-1}^{T}S_{k-1}^{T}AS_{k-1}\widehat{\Phi}_{k-1}]^{-1}\widehat{\Phi}_{k-1}^{T}\widehat{\Phi}_{k-1}\\
 & = & \widehat{\Phi}_{k-1}^{-1}[S_{k-1}^{T}AS_{k-1}]^{-1}\widehat{\Phi}_{k-1}.
\end{eqnarray*}
The residual relation becomes
\begin{eqnarray*}
Q_{k}\widehat{\Phi}_{k} & = & Q_{k-1}\widehat{\Phi}_{k-1}-AS_{k-1}\widehat{\Phi}_{k-1}\Upsilon_{k-1}\\
 & = & Q_{k-1}\widehat{\Phi}_{k-1}-AS_{k-1}[S_{k-1}^{T}AS_{k-1}]^{-1}\widehat{\Phi}_{k-1}.
\end{eqnarray*}
It yields
\[
Q_{k}\widehat{\Phi}_{k}\widehat{\Phi}_{k-1}^{-1}=Q_{k-1}-AS_{k-1}[S_{k-1}^{T}AS_{k-1}]^{-1}.
\]
Hence, $Q_{k}$ and $\widehat{\Phi}_{k}\widehat{\Phi}_{k-1}^{-1}$
are obtained by a QR factorization of the right-hand side.

\smallskip{}

\begin{algorithm}[ht]
\caption{Dubrulle-R BCG (DR-BCG)}
\label{alg-DBCG}

\begin{algorithmic}[1]

\STATE \textbf{input} $A$, $B$, $X_{0}$

\STATE $R_{0}=B-AX_{0}$

\STATE $[Q_{0},\widehat{\Phi}_{0}]=\mathtt{qr}(R_{0})$

\STATE $S_{0}=Q_{0}$

\FOR{$k=1,\dots$ until convergence}

\STATE $\widehat{\Pi}_{k-1}=(S_{k-1}^{T}AS_{k-1})^{-1}$

\STATE $X_{k}=X_{k-1}+S_{k-1}\widehat{\Pi}_{k-1}\widehat{\Phi}_{k-1}$

\STATE $[Q_{k},\widehat{\Psi}_{k}]=\mathtt{qr}(Q_{k-1}-AS_{k-1}\widehat{\Pi}_{k-1})$

\STATE $S_{k}=Q_{k}+S_{k-1}\widehat{\Psi}_{k}^{T}$

\STATE $\widehat{\Phi}_{k}=\widehat{\Psi}_{k}\widehat{\Phi}_{k-1}$

\ENDFOR

\end{algorithmic}
\end{algorithm}

\smallskip{}

Algorithm~\ref{alg-DBCG} describes Dubrulle's method with the notation
of this paper. In this algorithm we do not have to compute inverses
of $R_{k}^{T}R_{k}$ as in the standard BCG and there is no need for
deflation. However, it is not obvious that $S_{k-1}^{T}AS_{k-1}$
is always nonsingular. The properties of DR-BCG will be studied in
a forthcoming paper \cite{mt2025}.

\smallskip{}

To stop the algorithm we could compute $\widehat{\Phi}_{k}^{T}\widehat{\Phi}_{k}$.
The Gauss lower bounds of the error at iteration $k-1$ can be obtained
from
\begin{eqnarray*}
\Theta_{k-1}=R_{k-1}^{T}R_{k-1}\Upsilon_{k-1} & = & \widehat{\Phi}_{k-1}^{T}\widehat{\Phi}_{k-1}\widehat{\Phi}_{k-1}^{-1}[S_{k-1}^{T}AS_{k-1}]^{-1}\widehat{\Phi}_{k-1}\\
 & = & \widehat{\Phi}_{k-1}^{T}\widehat{\Pi}_{k-1}\widehat{\Phi}_{k-1},
\end{eqnarray*}
by computing the square roots of the diagonal entries. To compute
the quantity $\Theta_{k}^{(\mu)}=R_{k}^{T}R_{k}\Upsilon_{k}^{(\mu)}$
needed for the Gauss-Radau upper bound, we use \eqref{eq-Gammamu}
and replace $R_{k}^{T}R_{k}$ and $R_{k-1}^{T}R_{k-1}$ by $\widehat{\Phi}_{k}^{T}\widehat{\Phi}_{k}$
and $\widehat{\Phi}_{k-1}^{T}\widehat{\Phi}_{k-1}$.

\section{Numerical experiments}

\label{s-sec9}

For these experiments we consider examples were there is no rank deficiency
problems. BCG works well for these examples. The problems of rank deficiency
and the variants of BCG designed to avoid these problems will be considered
in \cite{mt2025}.

\smallskip{}

We first use the matrix corresponding to the discretization of the
Poisson equation with a 5-point finite difference scheme and a $30\times30$
mesh on the unit square. We have a matrix of order 900 and we took
10 random right-hand sides. The initial block vector is zero. The
value of $\mu$ is $0.0205$ when $\lambda_{1}\approx0.02052270$,
and the delay is equal to $1$ which means that, at iteration $k$,
we compute bounds at iteration $k-1$. For this example, the standard
BCG algorithm works well. The $A$-norm of the error for the first column is
shown in Figure~\ref{fig-BCG.1} as well as the Gauss and Gauss-Radau
bounds. We obtain very good bounds for the $A$-norm of the error
corresponding to the first column of the solution $X$. The Gauss
lower bound is very close to the exact $A$-norm. All columns converge
more or less similarly.

\smallskip{}

\begin{figure}[!htbp]
\centering \includegraphics[width=8cm]{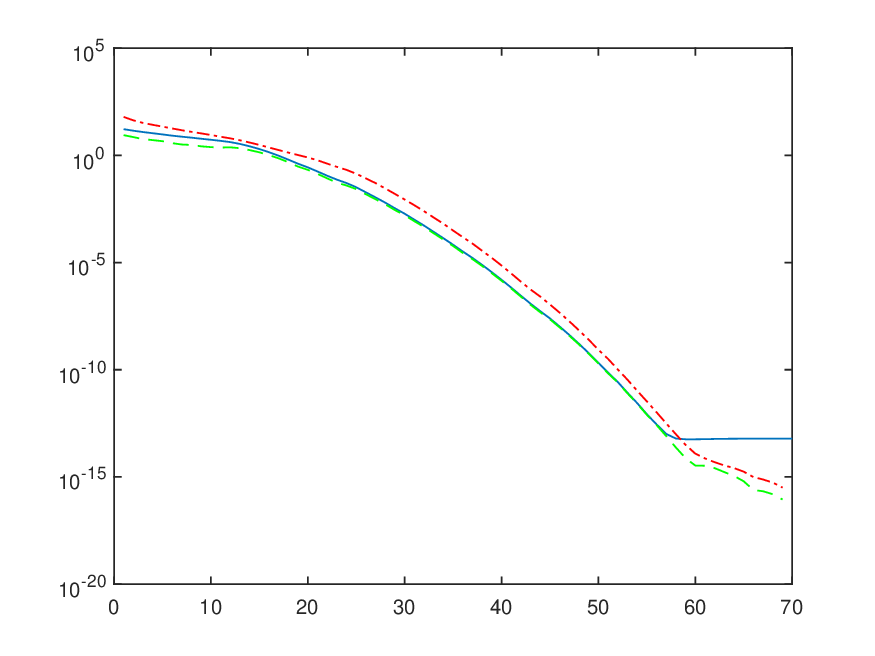} \caption{Poisson equation, standard BCG, first column, $A$-norm of the error
and bounds: Gauss (dashed green) and Gauss-Radau (dot-dashed red) }
\label{fig-BCG.1}
\end{figure}

\smallskip{}

The second example uses the matrix \texttt{bcsstk01}\footnote{available from https://sparse.tamu.edu}
of order $48$ with five random right-hand sides. We use $X_{0}=0$ and $\mu=3.417267\ 10^{3}$ when $\lambda_1\approx 3.41726756\ 10^3$.
Figure~\ref{fig-BCG.2} shows the results for standard BCG. What
is interesting to note is that it takes only 15 iterations to reach
the maximum attainable accuracy (which is not very good) when it
takes 160 iterations when using CG with the first column of $B$ as
right-hand side. With BCG, we do what can be considered as $15\times5=75$
matrix-vector products (even though it is 15 matrix-matrix products)
compared to 160 matrix-vector products for CG. At the beginning of
the iterations the Gauss lower bound is not very good, but it is tight
when the $A$-norm decreases fast. The bounds can be improved by looking
backwards for more than one iteration.

\smallskip{}

\begin{figure}[!htbp]
\centering \includegraphics[width=8cm]{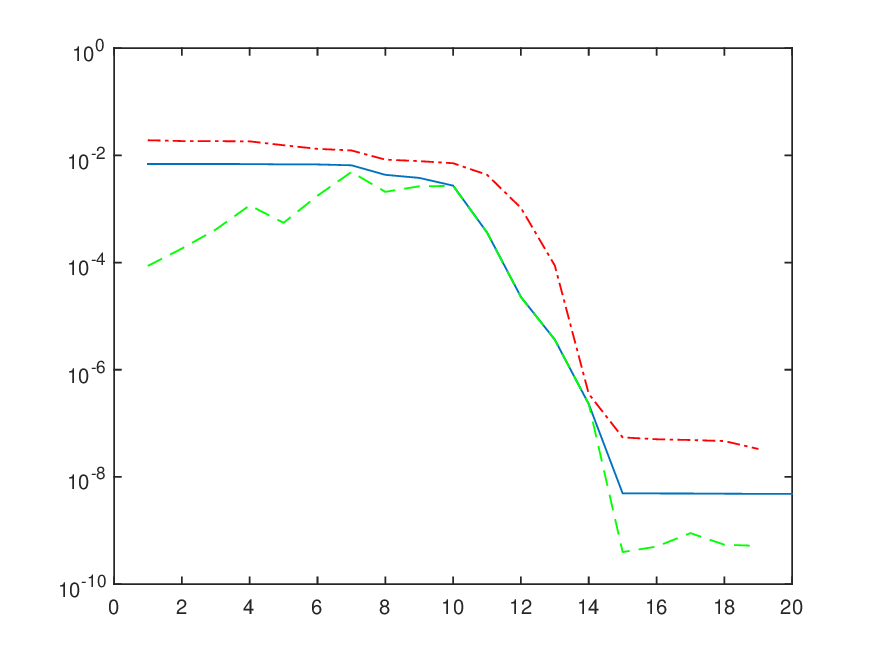} \caption{bcsstk01, standard BCG, first column, $A$-norm of the error and bounds:
Gauss (dashed green) and Gauss-Radau (dot-dashed red) }
\label{fig-BCG.2}
\end{figure}

\smallskip{}

For the third example, the matrix is \texttt{662\_bus} from the same
collection of matrices. It is of order 662. We use five random right-hand
sides, $X_{0}=0$ and $\mu=5\ 10^{-3}$ when $\lambda_1\approx 5.0472\ 10^{-3}$. Figure~\ref{fig-BCG.3} shows
the results for standard BCG. The Gauss-Radau upper bound is not very
tight, but this can be improved by taking a value of $\mu$ closer
to $\lambda_{1}$ or by looking backwards for more than one iteration.

\smallskip{}

\begin{figure}[!htbp]
\centering \includegraphics[width=8cm]{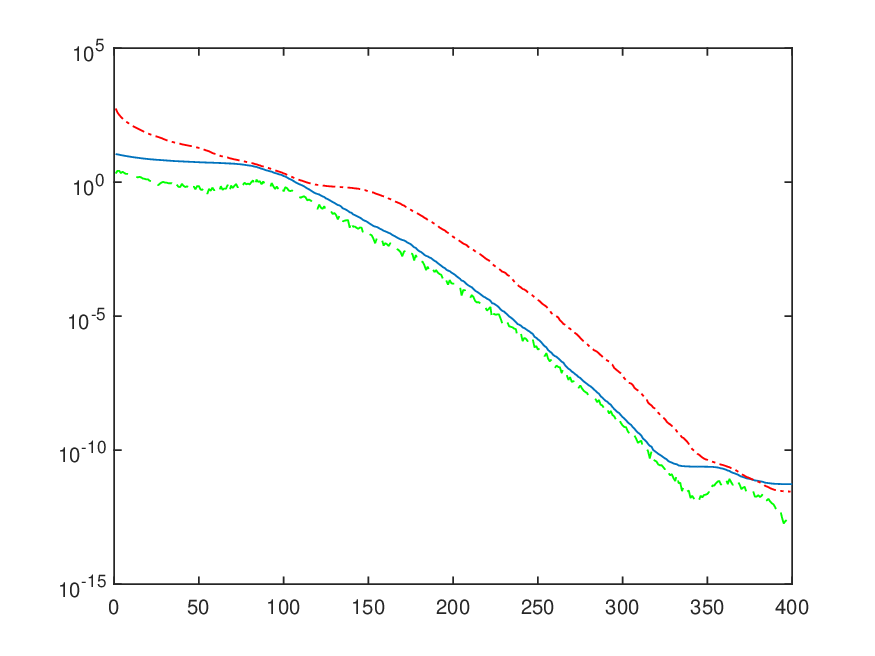} \caption{662\_bus, standard BCG, first column, $A$-norm of the error and bounds:
Gauss (dashed green) and Gauss-Radau (dot-dashed red) }
\label{fig-BCG.3}
\end{figure}

\smallskip{}

The last example uses the matrix \texttt{nos7} of order 729 and 10
random right-hand sides. We use $X_{0}=0$ and $\mu=4.1540\ 10^{-3}$  when $\lambda_1\approx 4.1541\ 10^{-3}$.
Figure~\ref{fig-BCG.3} shows the results for standard BCG. The lower
bound is oscillating because of the oscillations of the residual vector.
Despite that, the bounds decrease at the same rate as the $A$-norm
of the error.

\smallskip{}

\begin{figure}[!htbp]
\centering \includegraphics[width=8cm]{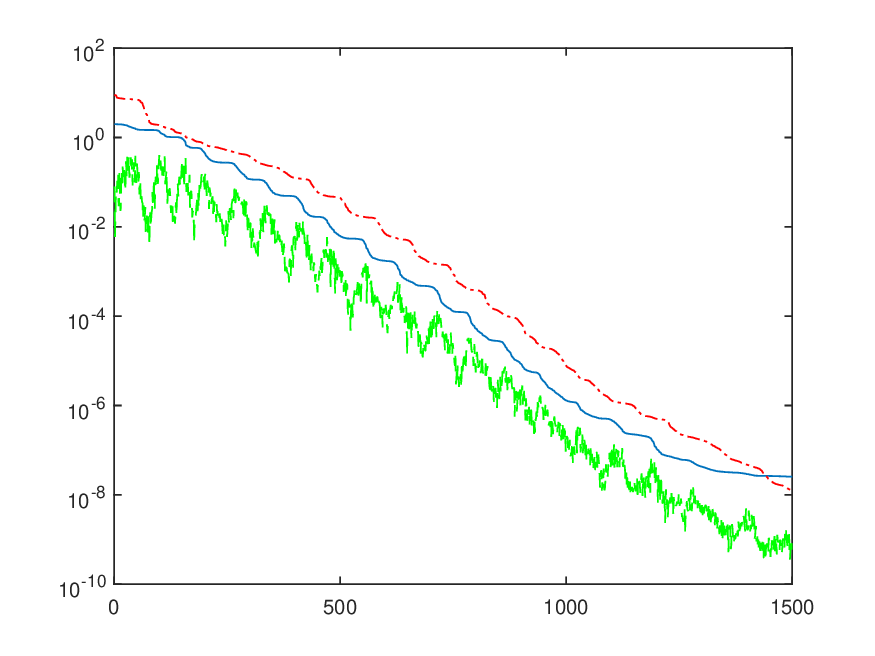} \caption{nos7, standard BCG, first column, $A$-norm of the error and bounds:
Gauss (dashed green) and Gauss-Radau (dot-dashed red) }
\label{fig-BCG.4}
\end{figure}

\smallskip{}

Numerical experiments with DR-BCG (Algorithm~\ref{alg-DBCG}) will be described in \cite{mt2025}.

\section{Conclusion}

\label{s-sec10}

In this paper we discussed in detail variants of the block conjugate gradient (BCG)
algorithm and obtained lower and upper bounds on the $A$-norm of
the error of each system. These bounds can be considered as generalizations
of the bounds known from the vector case. We derived them using
block Gauss and block Gauss-Radau quadrature rules. The bounds are
given in terms of the quantities computed in BCG as well as in terms
of the block tridiagonal matrices from the underlying block Lanczos
algorithm. We have shown how to compute these bounds in all considered
variants of BCG.

In generalizing the bounds to the block case, we faced several problems.
In particular, assuming that a positive underestimate to the smallest
eigenvalues of $A$ is available, it was necessary to prove that the remainder
of the block Gauss-Radau quadrature rule is a negative definite matrix.
This was shown algebraically, using Sylvester's law of inertia.

In this paper we have assumed that the block vectors are of full column
rank and that the coefficient matrices are nonsingular, so that all
iterations of BCG are well defined. In practical computations, 
the coefficient matrices in the classical versions of
BCG (Algorithm~\ref{alg-BCG} and Algorithm~\ref{alg-BCG-OL}) are
often close to singular. This can have a significant impact on the
efficiency of the algorithms in finite precision arithmetic, convergence
can be delayed, and the level of maximum attainable accuracy can be
poor. Deflation techniques can be used, but in our opinion they
have several disadvantages. The algorithms using deflation are
quite complicated and require the setting of various tolerances. Moreover,
their use in finite precision arithmetic is questionable; the theoretical
assumptions are often not satisfied, and the removal of vectors from
the process usually leads to a delay of convergence.

In our opinion, the most promising way to overcome the problem of
near singularity of the coefficient matrices (or rank deficiency of
the computed block vectors) is to use variants of BCG, as proposed
by A.~Dubrulle in \cite{dub2001}. In this paper we mentioned DR-BCG
(Algorithm~\ref{alg-DBCG}), which we consider one of the most important.
We plan to analyze this algorithm and some of its variants in our forthcoming work \cite{mt2025}.


\end{document}